\documentclass[11pt]{amsart}

\usepackage[a4paper,margin=1in]{geometry}
\usepackage{amsmath,amssymb,amsthm,mathtools}
\usepackage{microtype}
\usepackage{enumitem}
\usepackage{xurl}
\usepackage[colorlinks=true,linkcolor=blue,citecolor=blue,urlcolor=blue]{hyperref}

\numberwithin{equation}{section}

\newtheorem{theorem}{Theorem}[section]
\newtheorem{proposition}[theorem]{Proposition}
\newtheorem{lemma}[theorem]{Lemma}
\newtheorem{constructionlemma}[theorem]{Construction Lemma}
\newtheorem{corollary}[theorem]{Corollary}
\theoremstyle{remark}
\newtheorem{remark}[theorem]{Remark}

\newcommand{\W}{W}

\newcommand{\im}{\operatorname{im}}

\newcommand{\RR}{\mathbf R}

\newcommand{\m}{\mathfrak m}
\newcommand{\Atilde}{\widetilde A}
\newcommand{\Vtilde}{\widetilde V}
\newcommand{\Ktilde}{\widetilde K}
\newcommand{\stackstag}[1]{\href{https://stacks.math.columbia.edu/tag/#1}{Tag~#1}}

\title[Spherical Completeness and Coherence]{Spherical Completeness, Coherence, and GCD Properties of Formal Power Series and Witt Vector Rings}
\author{Yiding Wang}
\date{}


\begin{document}

\begin{abstract}
Let $K$ be a complete nonarchimedean valued
field with
$v(K^\times)=\RR$, and let $V=\mathcal O_K$. We prove that $K$ is
spherically complete if and only if $V[[T]]$ is coherent, and that this is
also equivalent to $V[[T]]$ being a GCD domain. If $K$ is perfect of
characteristic $p$, the same characterization holds for the Witt vector
ring $\W(V)$. Thus, this settles the previously unresolved full-real-value-group case in the coherence problems for both formal power series and Witt vector rings. 
In particular, this result also gives affirmative answers to Questions~9
and~10 of Anderson--Kang--Park. The proof combines a coherence criterion
for complete rings with a spherically complete valuation quotient and a
uniform construction of non-finitely generated intersections of two
principal ideals from an empty ball chain. 
\end{abstract}

\maketitle

\section{Introduction}

\subsection{Background and history}

Let $K$ be a complete nonarchimedean valued field and let
$V=\mathcal O_K$. If $V$ is a discrete valuation ring, then $V[[T]]$ is a
regular local ring and hence is both coherent and a UFD. The nondiscrete
case is subtler. Recall that a ring is coherent if every finitely
generated ideal is finitely presented.

The study of coherence for power series rings over valuation domains goes
back to J{\o}ndrup and Small \cite{JondrupSmall}. Anderson and Watkins subsequently
proved that if a rank-one valuation domain has value group a proper dense
subgroup of $\RR$, then its power series ring is not coherent
\cite{AndersonWatkins}; see also the recent homological treatment of Jones
\cite{JonesPowerSeries}. Thus, among nondiscrete rank-one valuation
domains, a positive result can occur only when the value group is all of
$\RR$.

The analogous factorization question was studied by
Anderson--Kang--Park. They proved that if $V$ is non-Noetherian and
$V[[T]]$ is a GCD domain, then $V$ has rank one and value group $\RR$.
They also constructed a complete valuation domain of value group $\RR$
whose power series ring is not a GCD domain
\cite[Theorems~1 and~8]{AndersonKangPark}. Their paper ends by asking
whether the valuation ring of the full Hahn field has a GCD power series
ring and, more generally, whether any non-Noetherian positive example
exists \cite[Questions~9 and~10]{AndersonKangPark}.

The same case remained unresolved for Witt vectors. If $K$ is perfect of
characteristic $p$, Kedlaya proved that $\W(V)$ is not coherent whenever
the value group is not isomorphic to $\RR$ as an ordered abelian group
\cite[Theorem~1.2, p.~130]{KedlayaAinf}. His proof treats separately the
case of a proper dense subgroup of $\RR$ and the case of a
non-Archimedean value group. The full real value group was left open
\cite[Remark~1.3, p.~131]{KedlayaAinf}.

\subsection{Main result}

Taken together, these results had left the same case unresolved in both
settings: complete valued fields with full real value group. Our main theorem
settles this case completely, identifying spherical completeness as the exact
dividing line.

\begin{theorem}\label{thm:classification}
Let $K$ be complete with $v(K^\times)=\RR$, and let $V=\mathcal O_K$.
\begin{enumerate}[label=\textup{(\roman*)},leftmargin=2em,itemsep=2pt,topsep=3pt]
\item In arbitrary characteristic, the following conditions are
equivalent: $K$ is spherically complete; $V[[T]]$ is coherent; and
$V[[T]]$ is a GCD domain.
\item If $K$ is perfect of characteristic $p$, the following conditions
are equivalent: $K$ is spherically complete; $\W(V)$ is coherent; and
$\W(V)$ is a GCD domain.
\end{enumerate}
\end{theorem}

\begin{samepage}
\begin{remark}
In particular, Theorem~\ref{thm:classification} answers Questions~9
and~10 of Anderson--Kang--Park affirmatively
\cite[Questions~9 and~10]{AndersonKangPark}. For Question~9, the full Hahn
field $S(\RR;\kappa)$ considered there is maximally complete
\cite[Theorem~1]{PoonenMaximal}, and therefore spherically complete
\cite[Theorem~4]{KaplanskyMaximal}. If $V_H$ denotes its valuation ring,
Theorem~\ref{thm:classification}\textup{(i)} shows that $V_H[[T]]$ is a
coherent GCD domain. Since its value group is $\RR$, the ring $V_H$ is
non-Noetherian. It consequently also supplies the positive
non-Noetherian example requested in Question~10.
\end{remark}
\end{samepage}

\subsection{Strategy and organization}

Section~\ref{sec:affine-cosets} records affine-coset intersection and
$\pi$-saturation results used throughout. Section~\ref{sec:kernel-lifting}
uses them to lift finite generation of kernels from $V$ to $W(V)$ or $V[[T]]$,
giving the positive implications. Section~\ref{sec:valuation-data}
extracts a normalized missing affine coset from an empty ball chain.
Section~\ref{sec:power} uses it to prove the converse for the power series case,
following Anderson--Kang--Park \cite{AndersonKangPark}, and
Section~\ref{sec:witt} adapts the construction to Witt vectors.

\section{Affine cosets and \texorpdfstring{$\pi$}{pi}-saturated ideals}
\label{sec:affine-cosets}

This section collects the lemmas used in both directions of the proof:
intersection results over spherically complete valuation domains and
principalization results for $\pi$-saturated ideals in complete rings.

All rings are commutative with identity. For $\pi\in A$ acting injectively
on an $A$-module $M$, a submodule $N\subseteq M$ is
\emph{$\pi$-saturated} if $N\cap\pi M=\pi N$, equivalently, if
$\pi x\in N$ implies $x\in N$.

We use the standard equivalent
characterizations of coherence: kernels of maps between finite free
modules are finitely generated, and finitely generated submodules of
finite free modules are finitely presented. We also use the resulting
fact that intersections of finitely generated submodules of a finite
free module are finitely generated
\cite[\stackstag{05CU}]{StacksProject}. A GCD domain is a domain in which
every pair of nonzero elements has a greatest common divisor.

We call a valuation ring $V$ \emph{$\RR$-valued complete} if its fraction field
$K$ is complete and the chosen valuation $v\colon K^\times\to\RR$ is surjective.

Every finitely generated torsion-free module over a valuation domain is
finite free: torsion-free modules are flat
\cite[\stackstag{0539}]{StacksProject}, and finite flat modules over local
rings are finite free \cite[\stackstag{00NZ}]{StacksProject}. If $M$ is finite free over $V$, call a subset $x+L$ an \emph{affine coset}
if $L\subseteq M$ is finitely generated, and call $L$ its \emph{direction}. Such cosets are convex in the
sense of \cite[Proposition~2.10]{ChernikovMennen}; the first lemma is therefore
a special case of \cite[Lemma~3.4]{ChernikovMennen}.

\begin{lemma}\label{lem:affine-compactness}
Let $V$ be a valuation domain whose fraction field $F$ is spherically complete. If $M$ is finite free over $V$ and
\[
Z_1\supseteq Z_2\supseteq Z_3\supseteq\cdots
\]
is a descending sequence of nonempty affine cosets in $M$, then $\bigcap_n Z_n\neq\varnothing$.
\end{lemma}

\begin{lemma}\label{lem:intersection-finiteness}
Let $V$ be a valuation domain with spherically complete fraction field
$F$ and value group $v(F^\times)=\RR$. If $M$ is finite free, then the
intersection of any nonempty family of finitely generated $V$-submodules
of $M$ is finite free.
\end{lemma}

\begin{proof}
Choose $M\cong V^d$, let $L=\bigcap_{i\in I}L_i$, and denote the maximal
ideal of $V$ by $\m$. Under the natural identification
$\operatorname{Hom}_V(\m,F)=F$, one has
$\operatorname{Hom}_V(\m,V)=V$: if $x\m\subseteq V$ and $v(x)<0$,
choose $a\in\m$ with $0<v(a)<-v(x)$, a contradiction. Each $L_i$ is
finite free, so, inside $F^d$,
\[
\operatorname{Hom}_V(\m,L)
=\bigcap_{i\in I}\operatorname{Hom}_V(\m,L_i)
=\bigcap_{i\in I}L_i=L.
\]
Since $L\subseteq V^d$ is bounded, Bhatt--Scholze's criterion
\cite[Lemma~6.14]{BhattScholzeWittGrassmannian} shows that $L$ is finite
free.
\end{proof}

The following lemma concerns $\pi$-saturated ideals in complete rings;
its second part is the argument underlying
\cite[Lemma~3]{AndersonKangPark}.

\begin{lemma}\label{lem:adic-basics}
Let $A$ be complete and separated for the $(\pi)$-adic topology, with
$\pi$ a nonzerodivisor and $A/\pi$ a valuation domain.

\noindent\textup{(i)}
If $H\bmod\pi\neq0$, then
\[
(H)\cap\pi^rA=\pi^r(H)\qquad(r\geq1).
\]
Consequently, $(H)$ is $\pi$-saturated and $\pi$-adically closed, and
division by $H$ on $(H)$ is continuous.

\noindent\textup{(ii)}
Let $J\subseteq A$ be $\pi$-saturated. If the image of $J$ in $A/\pi$ is
generated by the image of some $f\in J$, then $J=(f)$. In particular,
every finitely generated $\pi$-saturated ideal of $A$ is principal.
\end{lemma}

\begin{proof}
For \textup{(i)}, reduction modulo $\pi$ and cancellation of $\pi$ show
inductively that $Hy\in\pi^rA$ implies $y\in\pi^rA$. This proves the
displayed equality and saturation; separatedness also shows that $H$ is
regular. Moreover, a Cauchy sequence $(Hy_n)$ has Cauchy quotients
$(y_n)$, which proves closedness and continuity of division.

For \textup{(ii)}, given $x\in J$, saturation allows us to write
\[
x=fa_0+\pi x_1,\qquad x_1\in J.
\]
Iteration gives $a_i\in A$ and $x_n\in J$ such that
\[
x-f\sum_{i=0}^{n-1}\pi^ia_i=\pi^nx_n.
\]
The sum converges in $A$, and separatedness gives $x\in(f)$. The final
assertion follows because every finitely generated ideal of $A/\pi$ is
principal.
\end{proof}

\begin{corollary}\label{cor:intersection-principalization}
Let $A$ and $\pi$ be as in Lemma~\ref{lem:adic-basics}, let
$I=\bigcap_\lambda I_\lambda$ for a family of $\pi$-saturated ideals, and
write bars for reduction modulo $\pi$. If $H\in I$ and
$\bigcap_\lambda\overline{I_\lambda}=(\overline H)$, then $I=(H)$.
\end{corollary}

\begin{proof}
The ideal $I$ is $\pi$-saturated, and
$(\overline H)\subseteq\overline I\subseteq
\bigcap_\lambda\overline{I_\lambda}=(\overline H)$. Now apply
Lemma~\ref{lem:adic-basics}\textup{(ii)}.
\end{proof}

\section{Kernel lifting and coherence}\label{sec:kernel-lifting}

\begin{theorem}\label{thm:coherence-lifting}
Let $R$ be a commutative ring and let $\pi\in R$. Put $V=R/\pi$.
Assume that $\pi$ is a nonzerodivisor and that $R$ is $\pi$-adically
complete and separated. If $V$ is a valuation domain whose fraction field $F$ is
spherically complete and satisfies $v(F^\times)=\RR$, then $R$ is a
coherent GCD domain.
\end{theorem}

Applying Theorem~\ref{thm:coherence-lifting} to $(R,\pi)=(V[[T]],T)$ and,
in characteristic $p$, to $(R,\pi)=(\W(V),p)$ proves the positive
implications of Theorem~\ref{thm:classification}.

Before proving the theorem, we establish the lifting statement needed
for coherence. Retain its notation and put $R_n=R/\pi^n$.
Every nonzero $x\in R$ can be written as $x=\pi^rx_0$ with
$x_0\notin\pi R$. Thus, if $x=\pi^rx_0$ and $y=\pi^sy_0$ are
nonzero, then $\overline{x_0y_0}\neq0$ in $V$, so $xy\neq0$.
Hence $R$ is a domain; moreover, $(\pi)$ is prime because
$R/(\pi)=V$ is a domain.

Every $R_n$ is coherent. Indeed, $R_1=V$ is a coherent valuation domain, and
for $n\geq1$ the map $R_{n+1}\twoheadrightarrow R_n$ has square-zero kernel
$\pi^nR/\pi^{n+1}R$. Regularity of $\pi$ identifies this kernel, as an
$R_n$-module, with the finitely presented module $V=R_n/\pi R_n$. The
square-zero extension criterion \cite[Lemma~3.26]{BMS} therefore gives the
claim by induction.

Fix a finite matrix $\Phi\colon R^m\to R^n$ and write
\[
\mathcal K=\ker(\Phi),\qquad
\mathcal K_r=\ker(\Phi_r\colon R_r^m\to R_r^n).
\]
For $s\geq r$, let $\rho_{s,r}\colon\mathcal K_s\to\mathcal K_r$ be reduction
modulo $\pi^r$.

Since inverse limits commute with kernels, completeness and separatedness give
\begin{equation}\label{eq:kernel-limit}
\mathcal K\xrightarrow{\sim}\varprojlim_r\mathcal K_r.
\end{equation}

Define
\[
E_r=\bigcap_{s\geq r}\im(\rho_{s,r})\subseteq\mathcal K_r.
\]
Thus $E_r$ consists of solutions modulo $\pi^r$ that lift separately to every higher level, without an a priori compatible choice of lifts.

\begin{proposition}\label{prop:linear-compactness}
Every transition map $E_{r+1}\to E_r$ is surjective. Consequently,
\begin{equation}\label{eq:kernel-image}
\im(\mathcal K\to\mathcal K_1)
=\bigcap_{r\geq1}\im(\mathcal K_r\to\mathcal K_1).
\end{equation}
\end{proposition}

\begin{proof}
Fix $x_r\in E_r$. For $s\geq r+1$, put
\[
Y_s=\im(\mathcal K_s\to\mathcal K_{r+1}),\qquad
Z_s=\{z\in Y_s\mid\rho_{r+1,r}(z)=x_r\}.
\]
Each $Z_s$ is nonempty and the sequence is descending.

Let $N_r=\ker(\mathcal K_{r+1}\to\mathcal K_r)$. Coherence of $R_{r+1}$
shows that $\mathcal K_{r+1}$ is finitely presented. The same holds for
$\mathcal K_r$, first over $R_r$
and then over the finitely presented quotient
$R_r=R_{r+1}/(\pi^r)$. Hence $N_r$ is finitely presented. It is killed by
$\pi$; under the natural identification
$\pi^rR^m/\pi^{r+1}R^m\cong V^m$, it corresponds to a finite
torsion-free $V$-submodule and is therefore finite free.

Choose $z_s\in Z_s$ and set $D_s=Y_s\cap N_r$. Then
$Z_s=z_s+D_s$. The image $Y_s$ is finite over $R_{r+1}$, and coherence
shows that $D_s$ is finite; being killed by $\pi$, it is finite over $V$.
Fix $z^\circ\in Z_{r+1}$. After translating the common fiber
$\rho_{r+1,r}^{-1}(x_r)=z^\circ+N_r$ by $-z^\circ$, the sets $Z_s$ become
affine cosets in the finite free $V$-module $N_r$.
Lemma~\ref{lem:affine-compactness} gives a point of their intersection, which
belongs to $E_{r+1}$ and maps to $x_r$.

Starting with $x_1\in E_1$, recursively choose compatible $x_r\in E_r$. By
\eqref{eq:kernel-limit}, the system comes from an element of $\mathcal K$,
proving \eqref{eq:kernel-image}.
\end{proof}

\begin{proof}[Proof of Theorem~\ref{thm:coherence-lifting}]
Set $\Lambda_r=\im(\mathcal K_r\to\mathcal K_1)$. Each $\Lambda_r$ is a
finitely generated $V$-submodule of $\mathcal K_1\subseteq V^m$, and the
sequence is descending. By
Proposition~\ref{prop:linear-compactness} and
Lemma~\ref{lem:intersection-finiteness}, the module
$\im(\mathcal K\to\mathcal K_1)=\bigcap_r\Lambda_r$ in
\eqref{eq:kernel-image} is finitely generated over $V$.

The module $\mathcal K$ is $\pi$-saturated in $R^m$: if
$\pi x\in\mathcal K$, then $0=\Phi(\pi x)=\pi\Phi(x)$, and regularity of
$\pi$ on $R^n$ gives $\Phi(x)=0$. Hence
\[
\mathcal K\cap\pi R^m=\pi\mathcal K,
\qquad
\mathcal K/\pi\mathcal K\cong\im(\mathcal K\to\mathcal K_1).
\]
Moreover, $\bigcap_r\pi^r\mathcal K=0$ because $\mathcal K\subseteq R^m$.
The complete Nakayama lemma \cite[\stackstag{031D}]{StacksProject} now
shows that $\mathcal K$ is finitely generated. Hence $R$ is coherent.

It remains to prove the GCD assertion. For nonzero $f,g\in R$, write
\[
f=\pi^ma,\qquad g=\pi^nb,\qquad \pi\nmid a,b,
\]
and, after interchanging them if necessary, suppose that $m\leq n$.
Since $\pi$ is prime and $\pi\nmid a,b$, the ideals $(a)$ and $(b)$, and
hence their intersection, are $\pi$-saturated. Coherence of $R$ and
Lemma~\ref{lem:adic-basics}\textup{(ii)} therefore give
$(a)\cap(b)=(c)$ for some $c\in R$. We claim that
\[
(f)\cap(g)=(\pi^nc).
\]
The inclusion from right to left is immediate. Conversely, if
$x=\pi^mau=\pi^nbv\in(f)\cap(g)$, then cancellation of $\pi^m$ and
saturation of $(a)$ give
$bv\in(a)\cap(b)=(c)$, hence $x\in(\pi^nc)$. Thus
$(f)\cap(g)=(\pi^nc)$, and $fg/(\pi^nc)$ is a greatest common divisor of
$f$ and $g$. The case in which one element is zero is immediate.
\end{proof}

\begin{remark}\label{rem:sharp-inputs}
The proof uses two properties: descending intersections of finitely
generated submodules remain finitely generated, and descending affine
cosets meet. Let $\Gamma=v(F^\times)$. The first property for $M=V$
rules out higher rank: if $0<\delta$ lies in a nonzero proper convex
subgroup, the principal ideals of values $n\delta$ have a nonzero
nonprincipal intersection. After embedding an Archimedean $\Gamma$ in
$\RR$, it also rules out a nondiscrete proper subgroup, since values
increasing to any $\gamma\in\RR\setminus\Gamma$ yield principal ideals
with nonprincipal intersection. Thus $\Gamma$ is discrete or
isomorphic to $\RR$. In either case every ball chain has a countable
cofinal subchain; after rescaling, the second property for $M=V$ gives
spherical completeness. Hence, in the nondiscrete case, the two
properties recover precisely the hypotheses of
Theorem~\ref{thm:coherence-lifting}.
\end{remark}

\section{Ball chains and missing affine cosets}\label{sec:valuation-data}

From an empty ball chain in $K$, this section extracts rapidly decaying
valuation data and a missing affine coset over a maximal immediate
extension, to be used for $V[[T]]$ and $\W(V)$ in
Sections~\ref{sec:power} and~\ref{sec:witt}.

\begin{theorem}\label{thm:negative}
Let $V$ be an $\RR$-valued complete valuation ring with fraction field $K$,
and assume that $K$ is not spherically complete.
\begin{enumerate}[label=\textup{(\roman*)}]
\item In arbitrary characteristic, there exist $f,g\in V[[T]]$ such that
they have no greatest common divisor and $(f)\cap(g)$ is not finitely generated.
In particular, $V[[T]]$ is neither coherent nor a GCD domain.
\item If $K$ is perfect of characteristic $p$, there exist $f,g\in\W(V)$
such that they have no greatest common divisor and $(f)\cap(g)$ is not
finitely generated. In particular, $\W(V)$ is neither coherent nor a GCD domain.
\end{enumerate}
\end{theorem}

For the remainder of the paper, retain the notation and hypotheses of
Theorem~\ref{thm:negative}.

Fix a maximal immediate extension $\Ktilde/K$, and put
$\Vtilde=\mathcal O_{\Ktilde}$.
Here an extension of valued fields is called immediate if the induced maps
on value groups and residue fields are isomorphisms
\cite[p.~88]{PoonenMaximal}. Thus
$v(\Ktilde^\times)=v(K^\times)=\RR$, and the residue fields of
$\Ktilde$ and $K$ agree.
Such an extension exists by \cite[Corollary~6]{PoonenMaximal}, and
$\Ktilde$ is spherically complete by
\cite[Theorem~4]{KaplanskyMaximal}. 

An arbitrary empty ball chain need not satisfy the algebraic relations
among centers and directions required below. The following lemma replaces
it by missing affine cosets in a form compatible with the later algebra
and convergence arguments.

\begin{constructionlemma}[Rapidly decaying missing-coset data]
\label{constr:valuation-data}
There exist $c\in\mathfrak m_V\setminus\{0\}$,
$b_1\in\mathfrak m_{\Vtilde}\setminus K$, and sequences
$A_N\in\mathfrak m_V\setminus\{0\}$ and $s_N\in\mathfrak m_V$ with the
following properties.
\begin{enumerate}[label=\textup{(\roman*)}]
\item The affine cosets $s_N+A_N\Vtilde$ form a descending sequence and
satisfy
\begin{equation}\label{eq:missing-coset}
\bigcap_{N\geq1}(s_N+A_N\Vtilde)=b_1+c\Vtilde,
\qquad
(b_1+c\Vtilde)\cap K=\varnothing.
\end{equation}
\item There are elements $a_i\in\mathfrak m_V\setminus\{0\}$ such that
\[
A_N=\prod_{i=1}^N a_i,\qquad
s_N=-\sum_{i=1}^N\frac{c}{a_i}.
\]
\item Writing $\alpha_i=v(a_i)$, the sequence $(\alpha_i)$ is strictly
decreasing with limit zero and
\[
\sum_{i\geq1}r^i\alpha_i<\infty
\qquad\text{for every real number }r\geq1.
\]
Moreover,
\[
v(c)=\sum_{i\geq1}\alpha_i
\qquad\text{and}\qquad
v(A_N)=\sum_{i=1}^N\alpha_i.
\]
\end{enumerate}
\end{constructionlemma}

\begin{proof}
Choose a chain $\{B_\lambda\}$ of closed balls in $K$ with empty
intersection, and write $r_\lambda$ for their radii. These radii are
bounded above. Otherwise, a cofinal sequence of balls with radii tending
to $+\infty$ would have Cauchy centers, and their limit would belong to
every ball. If $\rho=\sup_\lambda r_\lambda$, no ball in the chain has
radius $\rho$. We may therefore choose a cofinal sequence
\[
B_n=z_n+\{x\in K\mid v(x)\geq\rho_n\},
\qquad B_{n+1}\subseteq B_n,
\]
where $(\rho_n)$ is strictly increasing with limit $\rho$.

The corresponding balls
\[
\widetilde B_n=z_n+\{x\in\Ktilde\mid v(x)\geq\rho_n\}
\]
are still nested. Since $\Ktilde$ is spherically complete, they have a
common point $\eta\in\Ktilde$. Necessarily $\eta\notin K$. By changing
each center without changing its ball, we may
also arrange that
\[
v(\eta-z_n)=\rho_n\qquad(n\geq0).
\]
Indeed, if the left side is larger than $\rho_n$, replace $z_n$ by
$z_n+t_n$ for an element $t_n\in K$ of valuation $\rho_n$.

Write $\epsilon_n=\rho-\rho_n$. After passing to a further cofinal
subsequence, we may assume that $(\epsilon_n)$ is strictly decreasing and
\[
0<\epsilon_n\leq 2^{-(n+1)^2}\qquad(n\geq0).
\]
In particular, if $D=\sum_{n\geq0}\epsilon_n$, then $D$ is finite.
Choose $\lambda\in K^\times$ with $v(\lambda)=D-\rho$ and replace
$\eta$ and all $z_n$ by their products with $\lambda$. We then have
\[
v(\eta-z_n)=D-\epsilon_n
\]
and
\begin{equation}\label{eq:normalized-empty-chain}
\bigcap_{n\geq1}
\{x\in K\mid v(x-z_n)\geq D-\epsilon_n\}=\varnothing.
\end{equation}

Choose $c\in K$ with $v(c)=D$, and for $i\geq1$ define
\[
\alpha_i=\epsilon_{i-1},
\qquad
a_i=\frac{c}{z_i-z_{i-1}}.
\]
The sequence $(\alpha_i)$ is strictly decreasing with limit zero, and its
Gaussian bound $\alpha_i\leq2^{-i^2}$ gives
$\sum_i r^i\alpha_i<\infty$ for every $r\geq1$. Since
$v(\eta-z_i)>v(\eta-z_{i-1})$, the strong triangle equality gives
\[
v(z_i-z_{i-1})=D-\alpha_i,
\qquad
v(a_i)=\alpha_i>0,
\qquad
\sum_{i\geq1}\alpha_i=D.
\]
Define $A_N$ and $s_N$ as in the statement, and put
$b_1=-(\eta-z_0)$. Since $D>\alpha_i>0$, the definitions give
$0\neq c,A_N\in\mathfrak m_V$ and $s_N\in\mathfrak m_V$.
Moreover, $v(b_1)=D-\alpha_1>0$ and $\eta\notin K$, so
$b_1\in\mathfrak m_{\Vtilde}\setminus K$. Also,
\[
v(b_1-s_N)=D-\alpha_{N+1}\geq v(A_N),
\]
so $s_N+A_N\Vtilde=b_1+A_N\Vtilde$ for every $N$.
Since $A_{N+1}=A_Na_{N+1}$ with $a_{N+1}\in\mathfrak m_V$, these affine
cosets form a descending sequence.

The sequence $v(A_N)=\sum_{i=1}^N\alpha_i$ increases to
$D=v(c)$. It follows that
\[
\bigcap_{N\geq1}A_N\Vtilde=c\Vtilde
\]
and hence that the first identity in \eqref{eq:missing-coset} holds.
Finally, if $x\in(b_1+c\Vtilde)\cap K$ and $y=z_0-x$, then
\[
v(y-z_N)=v(x-s_N)
\geq\min\{D,D-\alpha_{N+1}\}
=D-\epsilon_N.
\]
This contradicts \eqref{eq:normalized-empty-chain} and proves the second
identity.
\end{proof}

\begin{remark}[Bounded-radius ball chains]\label{rem:bounded-ball-chain}
In Construction Lemma~\ref{constr:valuation-data}, completeness of $K$
is used only to ensure that the valuation radii of the empty ball chain
are bounded above. Consequently, let $K$ be a possibly incomplete valued
field with $v(K^\times)=\RR$ and $V=\mathcal O_K$. If $K$ admits a
nested family of closed balls with empty intersection whose valuation
radii are bounded above, then the conclusion of Construction
Lemma~\ref{constr:valuation-data} still holds. Equivalently, for the norm
$|x|=\exp(-v(x))$, the ordinary radii of these balls are bounded away
from zero.

The arguments of Sections~\ref{sec:power} and~\ref{sec:witt} remain
valid under this hypothesis, with $K$ additionally assumed perfect of
characteristic $p$ in the Witt-vector case. Indeed, the maximal immediate
extension $\Ktilde$ is still spherically complete (and perfect in the
latter case); the valued-field completeness needed in those sections is
used over this extension. In the final noncoherence arguments,
Lemma~\ref{lem:adic-basics} applies to $V[[T]]$ with $\pi=T$ and to
$\W(V)$ with $\pi=p$, respectively: these rings are complete and
separated for the indicated adic topologies, the indicated element is a
nonzerodivisor, and the quotient is $V$. These properties do not require
valuation-topological completeness of $K$. Thus the conclusions of
Theorem~\ref{thm:negative}\textup{(i)} and, under the additional
perfect-characteristic-$p$ assumption, \textup{(ii)} still hold.
\end{remark}

\section{Formal power series}\label{sec:power}

Following \cite[Lemmas~6--7 and the proof of
Theorem~8]{AndersonKangPark}, we replace their special Hahn-series data
by $(a_i)$ and $(A_N,s_N)$ from
Construction Lemma~\ref{constr:valuation-data} and put
\[
A=V[[T]],
\qquad
\Atilde=\Vtilde[[T]],
\qquad
\xi_i=a_i-T,
\qquad
P_N=\prod_{i=1}^N(a_i-T).
\]
Propositions~\ref{prop:power-F} and~\ref{prop:power-divisor-data}
construct $F\in\Atilde$ and $f,g\in A$, respectively, and
Proposition~\ref{prop:power-gcd} proves that $F$ is a greatest common
divisor over $\Atilde$, whereas no greatest common divisor exists over $A$.

Lemma~\ref{lem:adic-basics}\textup{(i)} applies in $\Atilde$ with
$\pi=T$.

\begin{proposition}\label{prop:power-F}
For the elements $c,b_1$ in
Construction Lemma~\ref{constr:valuation-data}, there
exists $F\in\Atilde$ such that
\[
F\in\bigcap_{N\geq1}(P_N),
\qquad
F(T)=c+b_1T+O(T^2).
\]
Moreover,
\begin{equation}\label{eq:power-principal-intersection}
\bigcap_{N\geq1}(P_N)=(F).
\end{equation}
\end{proposition}

\begin{proof}
We first construct the required class modulo $T^2$. For each $N$,
\[
P_N=A_N\prod_{i=1}^N(1-T/a_i)
\equiv A_N\left(1-T\sum_{i=1}^Na_i^{-1}\right)\pmod{T^2}.
\]
If $P_NY$ has constant term $c$, then $Y$ has constant term $c/A_N$,
while the displayed congruence shows that the linear coefficient of $P_NY$
ranges over $s_N+A_N\Vtilde$.
By Construction Lemma~\ref{constr:valuation-data}, $b_1$ belongs to this
coset for every $N$, so $P_N\mid F_2=c+b_1T\in\Atilde/T^2$ for every
$N$.

Suppose $F_r\in\Atilde/T^r$ has been constructed. For fixed $N$, write
$F_r=P_NY_{N,r}$. The class of $Y_{N,r}$ modulo $T^r$ is uniquely
determined by $F_r$, and changing a lift by $T^rz$ changes the coefficient of
$T^r$ in the product by $A_Nz$. Hence the coefficients that extend $F_r$
modulo $T^{r+1}$ while preserving divisibility by $P_N$ form a nonempty
affine coset $C_{N,r}$ with direction $A_N\Vtilde$. Since
$P_N\mid P_{N+1}$, divisibility by $P_{N+1}$ implies divisibility by
$P_N$, and therefore $C_{N+1,r}\subseteq C_{N,r}$. Spherical
completeness gives a coefficient in $\bigcap_N C_{N,r}$; choosing it
defines $F_{r+1}$. Recursion produces compatible $F_r$ and hence
$F\in\Atilde$. For each fixed $N$, uniqueness makes the quotients
$Y_{N,r}$ compatible, and their limit satisfies $F=P_NY_N$.

The ideal $I=\bigcap_N(P_N)$ is $T$-saturated. Since
$\bigcap_N\overline{(P_N)}=\bigcap_NA_N\Vtilde
=c\Vtilde=(F\bmod T)$,
Corollary~\ref{cor:intersection-principalization} gives $I=(F)$.
\end{proof}

Let $\ell$ denote the residue characteristic when it is positive. Set
$e=1$ in residue characteristic zero and $e=\ell$ otherwise, and define
\[
N_i=ei+1,
\qquad
M_i=eN_i+1
\qquad(i\geq1).
\]
Both sequences are strictly increasing, $\gcd(N_i,M_i)=1$, and, in
positive residue characteristic, every $N_i$ and $M_i$ is prime to
$\ell$.

We now construct two infinite products; the $Q_n$ below collect the
factors of $f/F$ left after removing the distinguished factors $a_i-T$.

\begin{proposition}\label{prop:power-divisor-data}
Put
\[
\beta_1=\sum_{i\geq1}N_i\alpha_i<\infty,\qquad
\beta_2=\sum_{i\geq1}M_i\alpha_i<\infty,
\]
and choose $c_1,c_2\in K$ with $v(c_1)=\beta_1$ and
$v(c_2)=\beta_2$. For $n\geq1$, form the polynomials
\begin{align*}
f_n(T)&\coloneqq
\frac{c_1}{\prod_{i=1}^na_i^{N_i}}
\prod_{i=1}^n(a_i^{N_i}-T^{N_i})\in A,\\
g_n(T)&\coloneqq
\frac{c_2}{\prod_{i=1}^na_i^{M_i}}
\prod_{i=1}^n(a_i^{M_i}-T^{M_i})\in A.
\end{align*}
The sequences $(f_n)$ and $(g_n)$ converge $T$-adically in $A$; denote
their limits by $f,g\in A$. Their constant terms are $c_1,c_2$,
respectively, and in particular are nonzero.
\begin{enumerate}[label=\textup{(\roman*)},leftmargin=2em,itemsep=2pt,topsep=3pt]
\item One has $F\mid f,g$ in $\Atilde$.
\item Write $f=FG$. For $i\geq1$, let $\zeta_{N_i}$ be a nontrivial
$N_i$-th root of unity in a finite valued extension
$L_{N_i}/\Ktilde$, put $W_{N_i}=\mathcal O_{L_{N_i}}$, and set
$\pi_{i,\zeta_{N_i}}=T-a_i\zeta_{N_i}$. Then, in $W_{N_i}[[T]]$,
\[
\pi_{i,\zeta_{N_i}}\mid f,G,
\qquad
\pi_{i,\zeta_{N_i}}\nmid g,F.
\]
\item Define
\[
R_i(T)=\frac{a_i^{N_i}-T^{N_i}}{a_i-T}
=\sum_{r=0}^{N_i-1}a_i^{N_i-1-r}T^r,
\qquad
Q_n=\prod_{i=1}^nR_i(T).
\]
Then
\[
Q_n\mid G\quad(n\geq1),
\qquad
\bigcap_{n\geq1}(Q_n)=(G).
\]
\end{enumerate}
\end{proposition}

\begin{proof}
For \textup{(i)}, Construction
Lemma~\ref{constr:valuation-data}\textup{(iii)} makes
the sums defining $\beta_1$ and $\beta_2$ finite. The valuations of the
two normalizing scalars are the nonnegative tail sums
$\sum_{i>n}N_i\alpha_i$ and $\sum_{i>n}M_i\alpha_i$. Thus both scalars
lie in $V$ and hence $f_n,g_n\in A$. In $K[[T]]$, the same polynomials
can be rewritten as
\begin{align*}
f_n(T)&=c_1\prod_{i=1}^n\left(1-(T/a_i)^{N_i}\right),\\
g_n(T)&=c_2\prod_{i=1}^n\left(1-(T/a_i)^{M_i}\right).
\end{align*}
The new factors on the right are congruent to $1$ modulo $T^{N_i}$ and
$T^{M_i}$, respectively. Since all partial products lie in $A$, this
proves the asserted convergence in $A$. Passing to the limit also gives
the corresponding normalized product expressions for $f$ and $g$ in
$K[[T]]$, which will be used below. For each fixed $N$, one has
$P_N\mid f_n,g_n$ for all sufficiently large $n$. Closedness then gives
$P_N\mid f,g$, and Proposition~\ref{prop:power-F} gives $F\mid f,g$.

For \textup{(ii)}, write $f=FG$ and fix $i$ and $\zeta_{N_i}$ as in the
statement. Put $z=a_i\zeta_{N_i}$. The element
$z$ belongs to $\mathfrak m_{W_{N_i}}$, so evaluation on
$W_{N_i}[[T]]$ is continuous with prime kernel $(T-z)$. Thus
$\pi_{i,\zeta_{N_i}}=T-z\mid f_n$ for every $n\geq i$, so closedness
gives $\pi_{i,\zeta_{N_i}}\mid f$ in $W_{N_i}[[T]]$. On the
other hand, continuity of
evaluation gives
\[
g(z)=c_2\prod_{j\geq1}\left(1-(z/a_j)^{M_j}\right)\neq0.
\]
Indeed, for $j>i$,
$v((z/a_j)^{M_j})=M_j(\alpha_i-\alpha_j)\to+\infty$, so the tail product
and its inverse converge. The earlier factors are nonzero by their
valuations, and the $i$-th is nonzero because
$\zeta_{N_i}^{M_i}\neq1$ by $\gcd(N_i,M_i)=1$.
Thus $\pi_{i,\zeta_{N_i}}\nmid g$, and hence
$\pi_{i,\zeta_{N_i}}\nmid F$. Since $\pi_{i,\zeta_{N_i}}$ is prime,
$f=FG$ gives $\pi_{i,\zeta_{N_i}}\mid G$.

For \textup{(iii)}, fix $n$ and let $L_n/\Ktilde$ be a finite splitting
field for $\prod_{i\leq n}(X^{N_i}-1)$, equipped with the extended
valuation. It is
complete \cite[Theorem~3.1.2 and Proposition~1.2.2]{EnglerPrestel}. Put
$W_n=\mathcal O_{L_n}$. Since each $N_i$ is prime to the residue
characteristic, $X^{N_i}-1$ is separable, so
\[
R_i(T)=u_i
\prod_{\substack{\zeta_{N_i}^{N_i}=1\\ \zeta_{N_i}\neq1}}
(T-a_i\zeta_{N_i})
\]
for a unit $u_i\in W_n^\times$. In particular, $Q_n$ is, up to a unit, a
product of the pairwise nonassociate prime elements
$T-a_i\zeta_{N_i}$, where $1\leq i\leq n$ and $\zeta_{N_i}$ is a
nontrivial $N_i$-th root of unity. For different indices $i$, their roots
have the distinct valuations $\alpha_i$; for fixed $i$, if two factors were
associates, evaluating an association at either root would imply equality
of their roots. Since the factorization is squarefree, successive prime
cancellation gives $Q_n\mid G$ in $W_n[[T]]$.

Writing $G=Q_nY$ there, coefficient recursion from $Q_n(0)\neq0$ gives
$Y\in\Ktilde[[T]]\cap W_n[[T]]=\Atilde$. Thus $Q_n\mid G$ in $\Atilde$.

Put $J=\bigcap_{n\geq1}(Q_n)$.
Each $(Q_n)$ is $T$-saturated, hence so is $J$. Modulo $T$, the ideal
$(Q_n)$ becomes the principal ideal generated by
\[
Q_n(0)=\prod_{i=1}^n a_i^{N_i-1},
\]
whose valuation is the increasing partial sum
$\sum_{i=1}^n(N_i-1)\alpha_i$. In the valuation domain $\Vtilde$, the
intersection of these principal ideals is therefore the principal ideal
whose generator has the limiting valuation
\[
\sum_{i\geq1}(N_i-1)\alpha_i=\beta_1-v(c).
\]
The constant term of $G$ is $c_1/c$, which has exactly this valuation.
Thus $\bigcap_n\overline{(Q_n)}=(G\bmod T)$, and
Corollary~\ref{cor:intersection-principalization} gives $J=(G)$.
\end{proof}

\begin{proposition}\label{prop:power-gcd}
The element $F$ is a greatest common divisor of $f$ and $g$ in
$\Atilde=\Vtilde[[T]]$, whereas $f$ and $g$ have no greatest common divisor
in $A=V[[T]]$.
\end{proposition}

\begin{proof}
Let $d\mid f,g$, and write $f=dh$. Fix $n$, let
$L_n/\Ktilde$ be a finite splitting field for $Q_n$, and put
$W_n=\mathcal O_{L_n}$.
Proposition~\ref{prop:power-divisor-data} shows that every linear factor
$\lambda$ of $Q_n$ satisfies $\lambda\nmid g$, hence $\lambda\nmid d$.
Primality and successive cancellation give
$Q_n\mid h$ in $W_n[[T]]$. Writing $h=Q_nY$, the nonzero constant term of
$Q_n$ gives $Y\in\Ktilde[[T]]\cap W_n[[T]]=\Atilde$ by coefficient
recursion. Thus $h\in\bigcap_n(Q_n)=(G)$, say $h=Ge$, and cancellation in
$FG=f=dGe$ yields $F=de$.

Suppose now that $H\in A$ is a greatest common divisor. Then
$P_N\mid H$ for every $N$, so
$F\mid H$ by \eqref{eq:power-principal-intersection}, while
$H\mid F$ by the first assertion. Thus $H=uF$ for a unit
$u\in\Vtilde[[T]]^\times$.

From $H(0)=u(0)c\in K$ we obtain $u(0)\in V^\times$. After multiplying
$H$ by $u(0)^{-1}$, assume $u(0)=1$. Write
\[
H(T)=H(0)+H_1T+O(T^2),
\qquad
u(T)=1+u_1T+O(T^2).
\]
Since $F(T)=c+b_1T+O(T^2)$, comparison of linear coefficients in $H=uF$
gives $H_1=b_1+cu_1$. The left side belongs to $K$, while the right side lies in
$b_1+c\Vtilde$, contradicting the second identity in
\eqref{eq:missing-coset}.
\end{proof}

\begin{proof}[Proof of Theorem~\ref{thm:negative}\textup{(i)}]
Let $L=(f)\cap(g)$. Since $f(0),g(0)\neq0$,
Lemma~\ref{lem:adic-basics}\textup{(i)} shows that $L$ is $T$-saturated.
If it were finitely generated, Lemma~\ref{lem:adic-basics}\textup{(ii)}
would give $L=(h)$ for some $h\in A$. Write $fg=hd$. Since
$h\in(f)\cap(g)$, cancellation shows that $d\mid f,g$.
If $e$ is any common divisor, write $f=ef'$ and $g=eg'$. Then
$fg'=gf'\in(f)\cap(g)=(h)$, say $fg'=ht$. Comparing
$fg=hd$ with $fg=eht$ gives $d=et$. Thus $d$ is a greatest common
divisor, contradicting Proposition~\ref{prop:power-gcd}.
\end{proof}

\section{Witt vectors}\label{sec:witt}

Assume that $K$ is perfect of characteristic $p$. For any $x\in\Ktilde$,
adjoining a $p$-th root of $x$ gives an immediate extension because the
value group of $\Ktilde$ is divisible and its residue field is perfect.
Maximality therefore implies that $\Ktilde$ is perfect. Let
\[
A=\W(V),
\qquad
\Atilde=\W(\Vtilde).
\]
These rings are domains, and every element of $\Atilde$ can be uniquely written as
$\sum_{n\geq0}p^n[x_n]$; we call $x_n$ the $n$-th Teichm\"uller coordinate.

We follow the same divisor-descent pattern of
\cite[Lemmas~6--7]{AndersonKangPark}, using the data $(a_i)$ and
$(A_N,s_N)$ from Construction Lemma~\ref{constr:valuation-data} to define
the Witt-vector factors
\[
\xi_i=[a_i]-p,
\qquad
P_N=\prod_{i=1}^N\xi_i.
\]
Propositions~\ref{prop:witt-F} and~\ref{prop:witt-divisor-data}
construct $F\in\Atilde$ and $f,g\in A$, and
Proposition~\ref{prop:witt-gcd} proves that $F$ is their greatest common
divisor over $\Atilde$, whereas no greatest common divisor exists over $A$.

\begin{proposition}\label{prop:witt-F}
For the elements $c,b_1$ in
Construction Lemma~\ref{constr:valuation-data}, there
exists $F\in\Atilde$ such that
\[
F\in\bigcap_{N\geq1}(P_N),
\qquad
F=[c]+p[b_1]+O(p^2).
\]
Moreover,
\begin{equation}\label{eq:witt-principal-intersection}
\bigcap_{N\geq1}(P_N)=(F).
\end{equation}
\end{proposition}

\begin{proof}
We first construct the required class modulo $p^2$. For each $N$,
$P_N\bmod p=A_N$, and
\[
P_N=[A_N]\prod_{i=1}^N(1-p[a_i^{-1}])
\equiv[A_N]\left(1-p\sum_{i=1}^N[a_i^{-1}]\right)\pmod{p^2}.
\]
If the zeroth coordinate of $P_NY$ is $c$, then that of $Y$ is $c/A_N$,
while the displayed congruence shows that the first coordinate of $P_NY$
ranges over $s_N+A_N\Vtilde$.
By Construction Lemma~\ref{constr:valuation-data}, $b_1$ belongs to this
coset for every $N$, so $P_N\mid F_2=[c]+p[b_1]\in\Atilde/p^2$ for every
$N$. We now
construct compatible classes
\[
F_r\in\Atilde/p^r,
\qquad r\geq2,
\]
with the same property.

Suppose that $F_r$ has been constructed, and fix $N$. There is a unique
$y_{N,r}\in\Atilde/p^r$ such that $F_r=P_Ny_{N,r}$ by
Lemma~\ref{lem:adic-basics}\textup{(i)}. Choose a lift
$\widetilde y_{N,r}\in\Atilde/p^{r+1}$. All lifts of $y_{N,r}$ are
\[
\widetilde y_{N,r}+p^r[z],
\qquad z\in\Vtilde.
\]
Their products with $P_N$ are precisely the lifts of $F_r$ that remain in
$(P_N)$, and their $r$-th Teichm\"uller coordinates form an affine coset
$C_{N,r}$ with direction $A_N\Vtilde$, because $P_N\bmod p=A_N$. Since
$P_N\mid P_{N+1}$,
\[
C_{N+1,r}\subseteq C_{N,r}.
\]
Spherical completeness supplies a point of $\bigcap_NC_{N,r}$ and hence
$F_{r+1}$. The resulting system defines $F\in\Atilde$.

For fixed $N$, the unique classes $y_{N,r}$ are compatible; their inverse
limit $Y_N$ satisfies $F=P_NY_N$.

Let $I=\bigcap_N(P_N)$. It is $p$-saturated. Since
$\bigcap_N\overline{(P_N)}=\bigcap_NA_N\Vtilde
=c\Vtilde=(F\bmod p)$,
Corollary~\ref{cor:intersection-principalization} gives $I=(F)$.
\end{proof}

We now construct $f,g$; the $P_N$ record their common factors and the
$Q_N$ below the remaining multiplicities in $f/F$.

\begin{proposition}\label{prop:witt-divisor-data}
Choose an integer $q\geq2$ relatively prime to $p$, and put
\[
\beta_p=\sum_{i\geq1}p^i\alpha_i<\infty,\qquad
\beta_q=\sum_{i\geq1}q^i\alpha_i<\infty,
\]
and choose $c_p,c_q\in K$ with $v(c_p)=\beta_p$ and
$v(c_q)=\beta_q$. For $N\geq1$, form the elements
\begin{align*}
f_N&\coloneqq
\left[\frac{c_p}{\prod_{i=1}^Na_i^{p^i}}\right]
\prod_{i=1}^N([a_i]-p)^{p^i}\in A,\\
g_N&\coloneqq
\left[\frac{c_q}{\prod_{i=1}^Na_i^{q^i}}\right]
\prod_{i=1}^N\bigl([a_i]^{q^i}-p^{q^i}\bigr)\in A.
\end{align*}
The sequences $(f_N)$ and $(g_N)$ converge $p$-adically in $A$; denote
their limits by $f,g\in A$.
\begin{enumerate}[label=\textup{(\roman*)},leftmargin=2em,itemsep=2pt,topsep=3pt]
\item The reductions of $f$ and $g$ modulo $p$ are nonzero, and
$F\mid f,g$ in $\Atilde$.
\item Writing $f=FG$, for every $i\geq1$ one has
\[
\xi_i^{p^i}\mid f,\qquad
\xi_i\mid g,\quad \xi_i^2\nmid g,\qquad
\xi_i\mid F,\quad \xi_i^2\nmid F,
\qquad
\xi_i^{p^i-1}\mid G.
\]
\item For $Q_N\coloneqq\prod_{i=1}^N\xi_i^{p^i-1}$, one has
\[
Q_N\mid G\quad(N\geq1),
\qquad
\bigcap_{N\geq1}(Q_N)=(G).
\]
\end{enumerate}
\end{proposition}

\begin{proof}
For \textup{(i)}, the sums defining $\beta_p$ and $\beta_q$ are finite by
Construction Lemma~\ref{constr:valuation-data}\textup{(iii)}. The two
quotients occurring in the normalizing Teichm\"uller factors have the
nonnegative valuations
$\sum_{i>N}p^i\alpha_i$ and $\sum_{i>N}q^i\alpha_i$. They therefore lie
in $V$, which proves directly that $f_N,g_N\in A$. In $\W(K)$, the same
elements can be rewritten as
\begin{align*}
f_N&=[c_p]\prod_{i=1}^N(1-p[a_i^{-1}])^{p^i},\\
g_N&=[c_q]\prod_{i=1}^N
\left(1-p^{q^i}[a_i^{-q^i}]\right).
\end{align*}

The elementary congruence
$(1-px)^{p^i}\equiv1\pmod{p^{i+1}}$ holds in every commutative ring.
The $i$-th new factor of $g_N$ is congruent to $1$ modulo $p^{q^i}$.
Since all partial products lie in $A$ and
$A\cap p^r\W(K)=p^rA$, both sequences are Cauchy in $A$ and converge to
the elements $f$ and $g$ specified in the statement.

For fixed $N$, one has $P_N\mid f_M,g_M$ for all sufficiently large $M$.
Since $(P_N)$ is closed, one has
$P_N\mid f,g$ and hence $F\mid f,g$ by
Proposition~\ref{prop:witt-F}.
Moreover, $f\bmod p=c_p$ and $g\bmod p=c_q$ are nonzero.

For \textup{(ii)}, write $f=FG$. For fixed $i$, one has
$\xi_i^{p^i}\mid f_N$ whenever $N\geq i$; closedness of
$(\xi_i^{p^i})$ gives $\xi_i^{p^i}\mid f$.

For each $i$, equip $\Ktilde$ with the equivalent multiplicative norm
$\lvert x\rvert_i=p^{-v(x)/\alpha_i}$, so that
$\lvert a_i\rvert_i=p^{-1}$. Then $\xi_i=[a_i]-p$ is primitive of degree one
\cite[Definition~3.3.4]{KedlayaLiuFoundations}. Kedlaya--Liu
\cite[Theorem~3.3.7(a) and Example~3.3.8]{KedlayaLiuFoundations},
using Kedlaya
\cite[Theorem~5.11(a) and Remark~5.14]{KedlayaWittGeometry},
show that the quotient
$\mathcal B_i\coloneqq\Atilde/(\xi_i)$ is a valuation domain
carrying a multiplicative norm $\lvert\cdot\rvert_{\mathcal B_i}$ such that
$\lvert\theta_i([x])\rvert_{\mathcal B_i}=\lvert x\rvert_i$ for
$x\in\Vtilde$, where $\theta_i\colon\Atilde\to\mathcal B_i$ is the
quotient map. In particular, $(\xi_i)$ is prime. The additive valuation
$w_i=-\alpha_i\log_p\lvert\cdot\rvert_{\mathcal B_i}$ satisfies
\[
w_i(\theta_i([x]))=v(x)\quad(x\in\Vtilde),
\qquad
w_i(p)=v(a_i)=\alpha_i,
\]
since $p=\theta_i([a_i])$. These primes are pairwise nonassociate: if
$\xi_i$ and $\xi_j$ were associates, reduction by their common ideal would
give $[a_i]=p=[a_j]$, hence $\alpha_i=\alpha_j$ by the displayed valuation
formula. Since the $\alpha_i$ are strictly decreasing, $i=j$.

We next prove the stated multiplicities. Fix $i$ and put $m_j=q^j$. For
$N\geq i$, divide $g_N$ by $\xi_i$. Continuity of division shows that the
quotients form a $p$-adic Cauchy sequence and converge to $g/\xi_i$ in
$\Atilde$. Reduce them modulo $\xi_i$. Since
\[
\frac{[a_i]^{m_i}-p^{m_i}}{[a_i]-p}
=\sum_{r=0}^{m_i-1}[a_i]^{m_i-1-r}p^r,
\]
its image in $\mathcal B_i$ is the nonzero element $m_i[a_i]^{m_i-1}$, since
$m_i$ is prime to $p$. For $j<i$,
\[
w_i([a_j]^{m_j}-p^{m_j})=m_j\alpha_i,
\]
whereas for $j>i$ one has
\[
w_i([a_j]^{m_j}-p^{m_j})=m_j\alpha_j.
\]
The normalizing scalar has valuation $\sum_{j>N}m_j\alpha_j$. Hence the image of $g_N/\xi_i$ in $\mathcal B_i$ has the finite valuation
\[
\lambda_i=
\left(m_i-1+\sum_{j<i}m_j\right)\alpha_i
+\sum_{j>i}m_j\alpha_j,
\]
independent of $N\geq i$. The images converge in $\mathcal B_i$, and their
differences eventually have valuation at least $r\alpha_i$ for every $r$.
Taking $r\alpha_i>\lambda_i$ shows that the limit still has valuation
$\lambda_i$. Hence $\xi_i^2\nmid g$.

Since $F\mid g$ and $\xi_i\mid F$ while $\xi_i^2\nmid g$, write
$F=\xi_iF_i$ with $\xi_i\nmid F_i$. Cancelling the factor $\xi_i$
in the domain $\Atilde$ and repeatedly using the primality of $\xi_i$ gives
\[
\xi_i^{p^i-1}\mid G.
\]

For \textup{(iii)}, if $i\neq j$, primality and nonassociation give
$\xi_i\nmid\xi_j$. The same prime-power cancellation, applied successively
to the finitely many factors, therefore gives for every $N$
$Q_N\mid G$.
Put $J=\bigcap_{N\geq1}(Q_N)$.

Each $(Q_N)$ is $p$-saturated, hence so is $J$.
Modulo $p$, the ideal $(Q_N)$
maps to the principal ideal generated by
$\prod_{i\leq N}a_i^{p^i-1}$. Their intersection is generated by an element of
valuation
\[
\sum_{i\geq1}(p^i-1)\alpha_i=\beta_p-v(c).
\]
On the other hand, $G\bmod p=c_p/c$ has exactly this valuation.
Thus $\bigcap_N\overline{(Q_N)}=(G\bmod p)$, and
Corollary~\ref{cor:intersection-principalization} gives $J=(G)$.
\end{proof}

\begin{proposition}\label{prop:witt-gcd}
The element $F$ is a greatest common divisor of $f$ and $g$ in $\Atilde$,
whereas $f$ and $g$ have no greatest common divisor in $A=\W(V)$.
\end{proposition}

\begin{proof}
Let $d\mid f,g$, and write $f=dh$. Since $\xi_i^2\nmid g$,
either $\xi_i\nmid d$ or $d=\xi_id_i$ with $\xi_i\nmid d_i$.
Together with $\xi_i^{p^i}\mid f$, primality gives
$\xi_i^{p^i-1}\mid h$ in either case.
The $\xi_i$ are pairwise nonassociate, so $Q_N\mid h$ for every $N$.
Thus $h\in\bigcap_N(Q_N)=(G)$, say $h=Ge$, and cancellation in
$FG=f=dGe$ gives $F=de$.

Suppose now that $H\in A$ is a greatest common divisor. Then
$P_N\mid H$ for every $N$, so $F\mid H$ by
\eqref{eq:witt-principal-intersection}. Since $F$ is a greatest common
divisor in $\Atilde$, also $H\mid F$; hence $H=uF$ for some
$u\in(\Atilde)^\times$.

Write the first two Teichm\"uller coordinates of $H$ and $u$ as
\[
H=[H_0]+p[H_1]+O(p^2),
\qquad
u=[u_0]+p[u_1]+O(p^2),
\]
where $H_0,H_1\in V$ and $u_0,u_1\in\Vtilde$. Comparison of zeroth
coordinates gives $H_0=u_0c\in K$, so
$u_0\in V^\times$. After multiplying $H$ by $[u_0^{-1}]\in A^\times$,
and replacing $u$ accordingly, assume $u_0=1$. Modulo $p^2$,
\[
uF\equiv[c]+p\bigl([b_1]+[cu_1]\bigr)
\equiv[c]+p[b_1+cu_1]\pmod{p^2}.
\]
Here $p([x]+[y])\equiv p[x+y]\pmod{p^2}$, so
$H_1=b_1+cu_1$. But $H_1\in V\subseteq K$, while the right side belongs
to the missing coset $b_1+c\Vtilde$, contradicting the second identity in
\eqref{eq:missing-coset}.
\end{proof}

\begin{proof}[Proof of Theorem~\ref{thm:negative}\textup{(ii)}]
Let $L=(f)\cap(g)$. Since $f$ and $g$ have nonzero reductions,
Lemma~\ref{lem:adic-basics}\textup{(i)} shows that $(f)$ and $(g)$, and
hence $L$, are $p$-saturated. If $L$ were finitely generated,
Lemma~\ref{lem:adic-basics}\textup{(ii)} would make it principal; the
argument in the proof of
Theorem~\ref{thm:negative}\textup{(i)} would then give a greatest common
divisor, contradicting Proposition~\ref{prop:witt-gcd}.
\end{proof}

\begin{proof}[Proof of Theorem~\ref{thm:classification}]
The positive and negative implications are
Theorems~\ref{thm:coherence-lifting} and~\ref{thm:negative}, respectively.
\end{proof}

\section*{Acknowledgements}

The author thanks his advisor, Heng Du, for his guidance and helpful
discussions. He also thanks Adam Jones for the bounded-radius observation
recorded in Remark~\ref{rem:bounded-ball-chain}. The project was partially supported by the National Key R\&D
Program of China No.~2023YFA1009703, the National Natural Science Foundation
of China No.~12275066, and the Beijing Natural Science Foundation under Grant
No.~1254044.

\bibliographystyle{amsplain}
\bibliography{references}

\end{document}